\documentclass[a4paper,reqno]{amsart}

\usepackage{amsthm}
\usepackage{amssymb}
\usepackage{amsmath}
\usepackage{mathtools}
\usepackage{pdfpages}
\usepackage{tikz-cd}
\usepackage{exscale,cite,color,amsopn}
\usepackage[colorlinks=true,pdftex,unicode=true,linktocpage,bookmarksopen,hypertexnames=false]{hyperref}

\renewcommand{\leq}{\leqslant}
\renewcommand{\geq}{\geqslant}

\DeclareMathOperator{\ab}{ab}

\DeclareMathOperator{\Ann}{Ann}
\DeclareMathOperator{\Aut}{Aut}

\DeclareMathOperator{\Cen}{Z}
\DeclareMathOperator{\Fix}{Fix}
\DeclareMathOperator{\id}{id}

\DeclareMathOperator{\Ker}{Ker}
\DeclareMathOperator{\Rad}{Rad}

\DeclareMathOperator{\Soc}{Soc}

\newcommand{\N}{\mathbb{N}}
\newcommand{\Z}{\mathbb{Z}}

\makeatletter
\numberwithin{equation}{section}
\numberwithin{figure}{section}
\numberwithin{table}{section}
\newtheorem{thm}{Theorem}[section]

\newtheorem{cor}[thm]{Corollary}
\newtheorem{pro}[thm]{Proposition}

\theoremstyle{definition} 
\newtheorem{defn}[thm]{Definition}

\newtheorem{rem}[thm]{Remark}
\newtheorem{exa}[thm]{Example}

\makeatother

\title[Radical and weight of skew braces \ldots]{Radical and weight of skew braces and their
applications to structure groups of solutions of the Yang--Baxter equation}
\author{E. Jespers, {\L}. Kubat, A. Van Antwerpen, L. Vendramin}

\address[E. Jespers, {\L}. Kubat, A. Van Antwerpen]{Department of Mathematics, Vrije Universiteit Brussel, Pleinlaan 2, 1050 Brussel}
\email{eric.jespers@vub.be}
\email{lukasz.kubat@vub.be}
\email{arne.van.antwerpen@vub.be}

\address[L. Vendramin]{IMAS--CONICET and Departamento de Matem\'atica, FCEN, Universidad de Buenos Aires, Pabell\'on~1,
Ciudad Universitaria, C1428EGA, Buenos Aires, Argentina}
\email{lvendramin@dm.uba.ar}

\subjclass[2010]{Primary:16T25; Secondary: 16Nxx}
\keywords{Yang--Baxter equation, solution, skew brace, radical, weight, Wiegold problem, Schur theorem}

\begin{document}

\begin{abstract}
 %
 %
    We define the radical and weight of a skew left brace and provide some basic properties of these notions.
    In particular, we obtain a Wedderburn type decomposition for Artinian skew left braces. Furthermore, we prove analogues
    of a theorem of Wiegold, a theorem of Schur and its converse in the context of skew left braces. Finally, we apply these
    results to detect torsion in the structure group of a finite bijective non-degenerate set-theoretic solution of the
    Yang--Baxter equation. 
\end{abstract}

\maketitle

\section{Introduction}\label{sec:intro}

Radicals form an important tool in the structure of rings. The Jacobson radical was introduced by Jacobson in~\cite{MR12271}.
Later, Amitsur and Kurosh studied the general theory of radicals in~\cite{MR50563,MR59256,MR59257,MR0057236}.
Recently radical rings found new applications to the theory of the Yang--Baxter equation (YBE). 

The combinatorial version of the YBE has been studied intensively since the papers of Drinfeld~\cite{MR1183474},
Etingof, Schedler and Soloviev~\cite{MR1722951} and Gateva--Ivanova and Van den Bergh~\cite{MR1637256}. In~\cite{MR2278047} Rump found
an unexpected connection between involutive non-degenerate solutions and Jacobson radical rings as each radical ring produces an involutive non-degenerate solution.
More precisely, if $A$ is a radical ring, which means that the operation \[x\circ y=x+xy+y\] turns the ring $A$ into a group, then
\[r\colon A\times A\to A\times A,\quad r(x,y)=(-x+x\circ y,(-x+x\circ y)'\circ x\circ y),\]
where $z'$ denotes the inverse of the element $z\in A$ with respect to the circle operation, is an involutive  non-degenerate solution of the YBE. 

This is just the starting point of a recent new development in the theory of radical rings. One does not really need radical rings
to produce solutions of the YBE with Rump's formula. This is how one finds the notion of a left brace. Thus (left) braces appear naturally as the
generalizations of radical rings that allow us to produce solutions of the YBE by means of Rump's formula. Actually by a result of Rump~\cite{MR2278047}
radical rings correspond to two-sided braces, i.e., left braces that also are right braces. It turns out that, as
it was proved in~\cite{MR1722951}, all involutive non-degenerate solutions of the YBE are restrictions of those solutions obtained from left braces.
This means that left braces are the right algebraic structure needed to understand the combinatorial version of the YBE.
Since braces are groups with an additional abelian structure and radical rings provide natural examples of braces,
it is clear that methods of both group and ring theory will play a crucial role in investigations concerning braces.

In~\cite{MR3647970} braces were generalized to skew braces in order  to
study non-involutive bijective solutions of the YBE. Again one can use skew braces to restate the results of Soloviev~\cite{MR1809284} and
Lu, Yan and Zhu~\cite{MR1769723} related to non-involutive solutions of the YBE. Skew braces and some of their applications are intensively studied, see for example \cite{MR3982254, MR3834774, MR4017867, MR3763907, MR4033084, MR3977806,MR3814340}. 
For example Konovalov, Smoktunowicz and Vendramin  in \cite{KSV2018} initiated the investigations on the solvability of a brace via the introduction of some radicals, namely the Wedderburn radical and the solvable radical.
In this paper we continue along this line and introduce a new  radical, which leads to a nice description of the semisimple quotient.

Our radical has behaves nice as the quotient is a product of simples. We define 
the radical of a skew left brace as the intersection of all of its maximal ideals, and the weight 
as the minimal number of generators needed to generate the skew brace as an ideal. We explore
the connections between these notions in the context of skew braces. As an application
of our developments we prove, for example, 
that Artinian perfect skew left braces and finite skew left braces of square-free order both have always weight one. 

Our results on the structure of skew braces 
and some of the techniques developed have a natural application to solutions of the YBE. 
At the end of the paper, our tools are powerful enough to allow us to prove 
analogs of a group theoretical result of Schur and its
converse. These results 
are interesting on their own, but they are used in Theorem~\ref{thm:torsion} to answer
a question posed in~\cite{MR3974961} related to the torsion of the structure group of a finite solution. Several examples show
that our result on the torsion of the structure group of a solution is sharp. 

The paper is organized as follows. After giving some preliminaries in Section~\ref{sec:prel} we study in Section~\ref{sec:rad} a brace-theoretic radical akin to the Jacobson and Brown--McCoy radicals in ring theory, i.e., the intersection of maximal ideals of a skew left brace. Within group theory, that is, for trivial skew left braces, it corresponds to the Frattini subgroup, i.e., the intersection of all maximal subgroups. After proving some basic properties, we show in Theorem~\ref{thm:Gaschutz} an analog in the context of skew braces
of a well-known group-theoretic result of Gasch\"utz. In Theorem~\ref{thm:FP} we prove that under some
natural conditions about a skew left brace $A$, the quotient $A/\Rad(A)$ is a direct product of finitely many simple skew left braces. Since finiteness conditions are needed
for our purposes, we introduce Artinian and Noetherian skew left braces. In Section~\ref{sec:weight}, following Baer's paper~\cite{MR166267}
and motivated also by \cite{MOT2012}, we introduce the notion of weight, that is the minimal number of generators needed to generate a 
skew brace as an ideal, and study its basic properties. 
In Theorem~\ref{thm:Wiegold} we obtain a brace-theoretic analog of Wiegold's theorem: each Artinian perfect skew left brace has weight equal to one. In Theorem~\ref{thm:Kutzko} we prove 
an analog of a theorem of Kutzko~\cite{MR399272}, that 
for a skew left brace $A$ that satisfies some natural finiteness condition,
the weights of $A$ and $A/A^{(2)}$ coincide. As an application of all these results, we prove in Theorem~\ref{thm:square-free} that a finite skew left brace of square-free order has always weight one. In Section~\ref{sec:strbr} we provide some more applications of the notion of weight
of a skew left brace to study the structure group of a solution of the YBE. In particular, Theorem \ref{thm:Schur} is an analog of a celebrated group-theoretic result found by Schur. These results are then used in Section~\ref{sec:torsion} to study the torsion of the structure group of a solution of the YBE, see Theorem~\ref{thm:torsion}. Several examples show that the result of this theorem is sharp. 


\section{Preliminaries}\label{sec:prel}


In this section we give some of the needed background on skew left braces and their relation with the YBE (see for example \cite{MR3647970,MR1769723}).
A solution of the YBE is a pair $(X,r)$, where $X$ is a non-empty set and $r\colon X\times X\to X\times X$ is a bijective map such that 
\[r_1r_2r_1=r_2r_1r_2,\] where $r_1=r\times{\id}$ and $r_2={\id}\times r$. By convention, we always write
\[r(x,y)=(\sigma_x(y),\tau_y(x)),\quad x,y\in X.\] The solution $(X,r)$ is said to be non-degenerate if each $\sigma_x$ and
each $\tau_x$ is a bijective map. The solution is said to be involutive provided $r^2=\id_{X\times X}$. 

Etingof, Schedler and Soloviev defined in~\cite{MR1722951} the structure group of $(X,r)$ as the group $G(X,r)$ with generators
the elements of $X$ and relations given by $xy=uv$ whenever $r(x,y)=(u,v)$, that is 
\[G(X,r)=\langle X\mid xy=\sigma_x(y)\tau_y(x)\text{ for }x,y\in X\rangle.\]
This group turns out to be an important tool for studying solutions of the YBE. The structure group
of a solution is even more than just a group: it is a skew left brace. 

A skew left brace is a triple $(A,+,\circ)$, where $(A,+)$ and $(A,\circ)$ are (not necessarily abelian) groups such that
\[a\circ (b+c)=a\circ b-a+a\circ c\] for all $a,b,c\in A$. If $A$ is a skew left brace then the map
\[\lambda\colon (A,\circ)\to\Aut(A,+),\quad a\mapsto\lambda_a,\quad \lambda_a(b)=-a+a\circ b,\]
is a group homomorphism. The following formulas are easily verified:
\[a+b=a\circ\lambda^{-1}_a(b),\quad a\circ b=a+\lambda_a(b),\quad -a=\lambda_a(a'),\]
 where $a'$ denotes the inverse of $a$ in $(A,\circ)$. Moreover, the operation \[a*b=\lambda_a(b)-b=-a+a\circ b-b\] satisfies the following:
\begin{align}
    \begin{aligned}
        x*(y+z) & = x*y +y +x*z -y,\\
        (x\circ y)*z & = x*(y*z) + y*z + x*z
    \end{aligned}\tag{$*$}\label{*}
\end{align}
for all $x,y,z\in A$.
If the additive group  $(A,+)$ is commutative then $A$ is called a left brace.
Note that if $(G,+)$ is a group then $(G,+,+)$ is a skew left brace, called a trivial skew left brace.
A trivial skew left brace which is a left brace is simply called a trivial left brace.

Skew left braces yield solutions of the YBE. Indeed, if $A$ is a skew left brace, then the map $r_A\colon A\times A\to A\times A$,
defined as \[r_A(a,b)=(\lambda_a(b),\lambda_a(b)'\circ a\circ b),\] is a non-degenerate solution of the YBE.
This fact was first observed by Rump for involutive solutions (see \cite{MR2278047}).

As we mentioned in the introduction, radical rings are examples of skew left braces. Let us mention another very important example.
Namely, if $(X,r)$ is a non-degenerate solution of the YBE then there exists a unique skew left brace structure over the structure group $G(X,r)$
such that \[r_{G(X,r)}(\iota\times\iota)=(\iota\times\iota)r,\] where $\iota\colon X\to G(X,r)$ is the canonical map.
The multiplicative group of this skew left brace is isomorphic to the structure
group $G(X,r)$ and the additive group is isomorphic to the group
\[A(X,r)=\langle X\mid x+\sigma_x(y)=\sigma_x(y)+\sigma_{\sigma_x(y)}(\tau_y(x))\text{ for }x,y\in X\rangle,\]
already considered in the works of Soloviev~\cite{MR1809284} and Lu, Yan and Zhu~\cite{MR1769723}.
It is worth to mention that in general the map $\iota\colon X\to G(X,r)$ does not have to be an embedding; solutions for which this map is injective
are called injective solutions, as introduced in \cite{MR1809284}. However, if $\tilde{X}=\iota(X)$ and $\tilde{r}\colon\tilde{X}\times\tilde{X}\to\tilde{X}\times\tilde{X}$ is the
restriction of $r_{G(X,r)}$ then the pair $(\tilde{X},\tilde{r})$ is an injective solution of the YBE, called the injectivization of $(X,r)$,
such that $G(\tilde{X},\tilde{r})\cong G(X,r)$.

Let $A$ be a skew left brace. A left ideal of $A$ is an additive subgroup $I$ such that $\lambda_a(I)\subseteq I$ for all $a\in A$,
which is equivalent to $A*I\subseteq I$. It follows that $I$ is then also a subgroup of the multiplicative group of $A$. An ideal
is a left ideal that is normal both in the additive and the multiplicative group of $A$; the last condition is equivalent to demanding that
$I*A\subseteq I$. The socle of $A$ is defined as \[\Soc(A)=\Ker\lambda\cap\Cen(A,+)\] and the annihilator
of $A$ is defined \cite{MR3917122} as \[\Ann(A)=\Soc(A)\cap\Cen(A,\circ).\] Both the socle and the annihilator of $A$ are ideals of $A$.
Also the additive subgroup $A^{(2)}$ generated by all elements of the form $a*b$ for $a,b\in A$ is an ideal of $A$.
We say that $A$ is trivial if both operations on $A$, that is addition and multiplication, coincide or, in other words, if $A^{(2)}=0$.
It is clear that if $I$ is an ideal of $A$ then $A/I$ is trivial if and only if $A^{(2)}\subseteq I$.

\section{The radical of a skew brace}\label{sec:rad}

If $S$ is a subset of a skew left brace $A$ then $(S)$ denotes the ideal of $A$ generated by the set $S$,
that is the smallest ideal of $A$ containing $S$. We say that $S$ is a generating set for $(S)$.
Moreover, $A$ is said to be finitely generated provided there exists a finite subset $S$ of $A$ such that $A=(S)$.
Recall also that an ideal $M$ of $A$ is said to be maximal if $M$ is the only proper ideal of $A$ containing $M$.

\begin{defn}
    The radical $\Rad(A)$ of a skew left brace $A$ is defined as the intersection of all maximal
    ideals of $A$, if such exist, and $A$ otherwise.
\end{defn}

Clearly $\Rad(A)$ is an ideal of $A$. Furthermore, if $I$ is an ideal of $A$ contained in $\Rad(A)$ then $\Rad(A/I)=\Rad(A)/I$.
In particular, $\Rad(A/\Rad(A))=0$. Moreover, if $A$ has a non-zero homomorphic image which is finitely
generated (in particular, if $A$ is finitely generated) then it will be shown in Proposition~\ref{pro:weight} that $\Rad(A)\ne A$.
Note also that if $A$ is a trivial skew left brace then $\Rad(A)$ is the Frattini subgroup of $(A,+)$. Finally, in case $A$ is a left brace
corresponding to a radical ring, the radical $\Rad(A)$ is almost never equal to $A$, and thus, despite the similarity of its definition to the
Jacobson and Brown--McCoy radicals in ring theory, it behaves different. Therefore, $\Rad(A)$ should be treated as a new invariant of $A$.

\begin{defn}
    Let $A$ be a skew left brace. We say that $a\in A$ is a non-generating element provided $(S\cup\{a\})=A$
    implies $(S)=A$ for each subset $S\subseteq A$.
\end{defn}

\begin{defn}
    We say that an ideal $I$ of a skew left brace $A$ is small (or superfluous) if $I+J=A$ implies $J=A$ for each ideal $J$ of $A$.
\end{defn}

Using the above definitions we are able to give the following description of the radical of a skew left brace.

\begin{pro}\label{pro:desc}
    The radical $\Rad(A)$ of a skew left brace $A$ is equal to the set of all non-generating elements of $A$
    and also to the sum of all small ideals of $A$.
\end{pro}

\begin{proof}
    If $a\in A$ is a non-generating element of $A$ and $M$ is a maximal ideal of $A$ then $a\in M$, because otherwise $(M\cup\{a\})=A$
    would lead to $M=A$, a contradiction. Hence $a\in\Rad(A)$. Also conversely, let $b\in\Rad(A)$ and $(S\cup\{b\})=A$ for some $S\subseteq A$.
    We claim that $b\in(S)$, which clearly implies $(S)=A$. To prove the claim suppose, on the contrary, that
    $b\notin(S)$ and consider the family \[\mathcal{F}=\{I\lhd A:(S)\subseteq I\text{ but }b\notin I\}.\]
    Applying Zorn's Lemma to $\mathcal{F}$ we see that there exists a maximal ideal $M$ of $A$ such that $(S)\subseteq M$ but $b\notin M$,
    a contradiction with $b\in\Rad(A)$. This completes the proof of the first part of the Proposition.
        
    If $I$ is a small ideal of $A$ and $M$ is a maximal ideal of $A$ then necessarily $I\subseteq M$, because otherwise $I+M=A$ would lead to
    $M=A$, a contradiction. Hence the sum of all small ideals of $A$ is contained in $\Rad(A)$. On the other hand, if $c\in\Rad(A)$ then a similar
    argument as in the first paragraph of the proof shows that the ideal $(c)$ is small, and the result follows.
\end{proof}

The following result is similar to that of Gasch\"utz~\cite{MR57873}, which says that for a finite group $G$ and its maximal
subgroup $M$ we have $[G,G]\subseteq M$ or $\Cen(G)\subseteq M$. In particular, the intersection $[G,G]\cap\Cen(G)$ is contained in the Frattini
subgroup $\Phi(G)$ of $G$. We denote the commutator subgroup of the group
$(A,+)$ as $[A,A]_+$. Note that if $(A,+,\circ)$ is a skew left brace, then
$[A,A]_+$ is a strong left ideal of $A$ (see \cite{JKVAV} for the
definition of strong left ideals in skew braces).

\begin{thm}\label{thm:Gaschutz}
	Let $A$ be a skew left brace. If $M$ is a maximal ideal of $A$ then $[A,A]_+ +A^{(2)}\subseteq M$ or $\Soc(A)\subseteq M$. In particular,
	$A^{(2)}\cap\Soc(A)\subseteq\Rad(A)$.
\end{thm}

\begin{proof}
    Let $B=A/M$. Clearly $(\Soc(A)+M)/M\subseteq\Soc(B)$. As $B$ is a simple skew left brace, either $\Soc(B)=0$ or $\Soc(B)=B$.
    In the former case we get $\Soc(A)\subseteq M$. Whereas, if $\Soc(B)=B$ then $(A^{(2)}+M)/M=B^{(2)}=0$ yields $A^{(2)}\subseteq M$.
    Furthermore, as $\Soc(B)=B$, it follows that $\Cen(B,+) = B$. Hence, $[A,A]_+ \subseteq M$.
    Since $A^{(2)} \cap \Soc(A) \subseteq M$ for each maximal ideal $M$ of $A$, it follows that $A^{(2)}\cap\Soc(A)\subseteq\Rad(A)$.
\end{proof}

A skew left brace $A$ is said to be prime if for all non-zero ideals $I$ and $J$ one has $I*J\ne 0$.
An ideal $P$ of a skew left brace $A$ is said to be prime if $A/P$ is a prime skew left brace. 

\begin{pro}\label{pro:NP}
    Let $M$ be a maximal ideal of a skew left brace $A$. Then $M$ is a prime ideal if and only if $A^{(2)}\nsubseteq M$.
\end{pro}

 \begin{proof}
    If $A^{(2)}\subseteq M$ then $M$ cannot be prime, because otherwise we would get $A\subseteq M$, a contradiction.
    On the other hand, if $M$ is not a prime ideal then there exist ideals $I,J$ of $A$ such that $I*J \subseteq M$
    but $I,J\nsubseteq M$. Then clearly $I + M = J + M = A$ and thus, by \eqref{*}, we get
    \begin{align*}
        A^{(2)} & =(I+M)*(J+M)\subseteq(I\circ M)*J+M\\
        & \subseteq I*(M*J)+M*J+I*J+M\subseteq M.\qedhere
    \end{align*}
\end{proof}

Recall that a skew left brace $A$ is perfect if $A=A^{(2)}=A*A$.

\begin{cor}
    Let $A$ be a finitely generated skew left brace. Then $A$ is perfect if and only if all maximal ideals of $A$ are prime.
\end{cor}

It is worth to mention that maximal non-prime ideals of skew braces do exist. For example, if $A$ is a finite two sided brace, with more than one element,
i.e., a non-trivial finite nilpotent ring, then each maximal ideal of $A$ cannot be prime.
In \cite[Example 5.3]{CJO18} an example is given of a finite prime non-simple left brace $A$  with a unique non-trivial ideal, say $M$. Moreover, as a left brace, $M$ is simple. Hence $\Rad(A)=M$ is perfect; again confirming that the behaviour of the brace-theoretic radical can be very different from that  in ring theory. 

\begin{defn}
    If $A$ is a skew left brace then $\Rad'(A)$ denotes the intersection of all maximal ideals of $A$ which are prime,
    if such exist, and $A$ otherwise. Clearly $\Rad(A)\subseteq\Rad'(A)$.
\end{defn}

\begin{pro}\label{pro:inc}
	Assume that $I$ is an ideal of a skew left brace $A$. Then $\Rad(I)\subseteq\Rad(A)$ and $\Rad'(I)\subseteq\Rad'(A)$.
\end{pro}

\begin{proof}
    If $\Rad(A)=A$ then clearly $\Rad(I)\subseteq\Rad(A)$. So assume that $\Rad(A)\ne A$ and choose a maximal ideal $M$ of $A$.
    If $I\subseteq M$ then $\Rad(I)\subseteq M$. Whereas, if $I$ is not contained in $M$ then $I+M=A$ and thus
    \[I/(I\cap M)\cong(I+M)/M=A/M\] is a simple skew left brace. Therefore, $I\cap M$ is a maximal ideal of $I$.
    Hence we get $\Rad(I)\subseteq I\cap M\subseteq M$. Since a similar proof works for $\Rad'$, the result follows.
\end{proof}

We say that a skew left brace $A$ is Artinian if every descending chain of ideals of $A$ is eventually stationary.
The following result may be viewed as a brace-theoretic analog of the celebrated Artin--Wedderburn decomposition theorem for
semisimple (that is, Artinian and semiprimitive) rings.

\begin{thm}\label{thm:FP}
    If $A$ is an Artinian skew left brace (in particular, if $A$ is finite) then the skew left brace $A/\Rad(A)$
    is isomorphic to a direct product of finitely many simple skew left braces. In particular, if furthermore $A$ 
    is prime, then $A$ is simple if and only if $\Rad (A)=0$.
\end{thm}

\begin{proof}
    Consider the family \[\mathcal{F}=\{M_1\cap\dotsb\cap M_n:\text{each }M_i\text{ is a maximal ideal of }A\text{ and }n\geq 1\}.\]
    Since $A$ is Artinian, there exists a minimal member $R=M_1\cap\dotsb\cap M_n$ of $\mathcal{F}$. First, note that $\Rad(A)=R$.
    Indeed, clearly $\Rad(A)\subseteq R$ and if $M$ is a maximal ideal of $A$ then $R\cap M$
    is a member of $\mathcal{F}$. Hence $R=R\cap M\subseteq M$ by the minimality of $R$ in $\mathcal{F}$.
    Because $R\subseteq M$ for each maximal ideal $M$ of $A$, we get $R\subseteq\Rad(A)$.
    
    Next, we claim that if none of the ideals $M_1,\dotsc,M_n$ is redundant in the presentation of the radical $R=M_1\cap\dotsb\cap M_n$
    as an intersection of maximal ideals then $A/R\cong A/M_1\times\dotsb\times A/M_n$. 
    If $n=1$ then the claim is trivial. If $n=2$ then the natural injective morphism $A/R\to A/M_1\times A/M_2$ is also surjective.
    Indeed, if $a_1,a_2\in A$ then using $A=M_1+M_2$ we may write $a_i=x_i+y_i$, where $x_i\in M_1$ and $y_i\in M_2$ for $1\leq i\leq 2$.
    Define $x=y_1+x_2$. Then
    \begin{align*}
        x-a_1 & =y_1+x_2-y_1-x_1\in M_1,\\
        x-a_2 & =y_1+x_2-y_2-x_2\in M_2.
    \end{align*}
    Hence the claim also follows in this case.
    (Note that the above proof shows in fact that $A/(I\cap J)\cong A/I\times A/J$ for any ideals $I$ and $J$ of $A$ satisfying $I+J=A$.)
    Finally, if $n>2$ then $J=M_2\cap\dotsb\cap M_n\nsubseteq M_1$ (otherwise $M_1$ would be redundant in the presentation of $R$ as
    an intersection of maximal ideals) and thus we get $M_1+J=A$. Hence (by the comment above) $A/R\cong A/M_1\times A/J$. Since $A/J$
    is an Artinian skew left brace with trivial radical that can be presented as an intersection of $n-1$ maximal ideals $M_2/J,\dotsc,M_n/J$,
    it follows by induction on $n$ that \[A/J\cong (A/J)/(M_2/J)\times\dotsb\times (A/J)/(M_n/J)\cong A/M_2\times\dotsb\times A/M_n,\]
    which completes the proof.
\end{proof}

Solvable braces were introduced in~\cite{BCJO2}, and studied later in \cite{KSV2018}.
For a skew left brace $A$ define $A_1=A$ and $A_{i+1}=A_i*A_i$ for $i\geq 1$.
Then $A$ is said to be solvable provided $A_n=0$ for some $n\geq 1$.  

\begin{pro}\label{pro:A2}
	Assume that $A$ is a solvable skew left brace. If $M$ is a maximal ideal of $A$ then $A^{(2)}\subseteq M$, so $M$ is not prime.
	In particular, $A^{(2)}\subseteq\Rad(A)$.
\end{pro}

\begin{proof}
    Since $A$ is a solvable skew left brace, it follows that $A/M$ is a solvable simple skew left brace.
    Hence $(A/M)^{(2)} = 0$ and thus we get $A^{(2)}\subseteq M$. Moreover, the remaining assertion is a
    consequence of Proposition~\ref{pro:NP}.
\end{proof}

The following examples show that the converse of Proposition~\ref{pro:A2} is not true.

\begin{exa}
	Let $A=S(24,708)$, a non-solvable skew left brace with additive group isomorphic to $S_4$, 
	see~\cite[\S2]{KSV2018}. 
	A straightforward computer
	calculation shows that $A$ has a unique, hence maximal, ideal $0\ne I\ne A$ and $0\ne A^{(2)}\ne A$.
	Therefore $A^{(2)}=I=\Rad(A)$.
\end{exa}

\begin{exa}\label{exa:odd}
    Let $A$ be a simple left brace of odd order, which admits an automorphism of order $2$. Examples of simple left braces
    of odd order were constructed in \cite{CJO18}. Moreover, since some of these left braces are the so-called asymmetric products
    of trivial braces, it is easy to check that they do admit automorphisms of order $2$. Then the left brace $B= A \rtimes C_2$
    is a non-solvable brace, for which $B^{(2)} $ is contained in every maximal ideal.
\end{exa}


\section{The weight of a skew brace}\label{sec:weight}

In this Section, following Baer's paper~\cite{MR166267} and motivated also by \cite{MOT2012}, we introduce the notion of weight
of a skew left brace, that is the minimal number of generators needed to generate a skew left brace as an ideal, and study its basic properties.

\begin{defn}
    Let $A$ be a non-zero skew left brace. The weight $\omega(A)$ of $A$ is defined as the minimal number of elements
    of $A$ needed to generate $A$ (as an ideal). By convention, we put $\omega(A)=1$ if $A=0$. 
\end{defn}

\begin{exa}
    Let $p$ be a prime number and $n\in\N$. The trivial brace over the elementary abelian group $C_p^n$ has weight $n$. 
\end{exa}

Using standard arguments, including those based on Zorn's lemma, it is easy to show the following facts.

\begin{pro}\label{pro:weight}
    Let $A$ be a skew left brace.
    \begin{enumerate}
        \item If $I$ is a proper ideal of $A$ such that $A/I$ has finite weight then there exists a maximal ideal $M$ of $A$ containing $I$.
        \item If $A$ is of finite weight and $S$ is a subset of $A$ then $S$ generates $A$ if and
    only if its image in $A/\Rad(A)$ generates $A/\Rad(A)$. In particular, \[\omega(A)=\omega(A/\Rad(A)).\]
    \end{enumerate}
\end{pro}

Our next result may be treated as a brace-theoretic version of Wiegold's problem that
appeared in the Kourovka Notebook \cite[Problem 5.52]{Kourovka}, which asks whether
a finitely generated perfect group is the normal closure of a single element.

\begin{thm}\label{thm:Wiegold}
    Assume that $A$ is an Artinian skew left brace of finite weight. If $A$ is perfect then $\omega(A)=1$.
\end{thm}

\begin{proof}
    Clearly we may assume that $A\ne 0$. Then, by Proposition \ref{pro:weight}, we may also assume that $\Rad(A)=0$.
    Furthermore, Theorem \ref{thm:FP} assures that we may restrict
    to the case $A=A_1\times\dotsb\times A_n$, where $A_1,\dotsc,A_n\ne 0$ are Artinian, simple and perfect skew left braces.
    Because $A_i*A_i=A_i\ne 0$, there exists $x_i\in A_i$ such that $A_i*x_i\ne 0$. Let $x=(x_1,\dotsc,x_n)\in A$ and let $I=(x)$ be the
    ideal of $A$ generated by $x$. Since $a_i*x_i\ne 0$ for some $a_i\in A_i$, we get
    \[e_i=(0,\dotsc,0,a_i,0,\dotsc,0)*x=(0,\dotsc,0,a_i*x_i,0,\dotsc,0)\in I.\]
    But $(a_i*x_i)=A_i$ guarantees that the ideal of $A$ generated by $e_i$ is equal to
    \[\{0\}\times\dotsb\times\{0\}\times A_i\times\{0\}\times\dotsb\times\{0\}.\]
    Hence $A=\sum_{i=1}^n(e_i)\subseteq I$ and thus $A=I=(x)$.
\end{proof}


The following is similar to the result of Kutzko \cite{MR399272}, which asserts that the group-theoretic weight
of a group $G$ satisfying a certain mild finiteness condition is equal to the weight of its abelianization $G_{\ab}=G/[G,G]$.

\begin{thm}\label{thm:Kutzko}
    Let $A$ be a skew left brace of finite weight. If the lattice of ideals of $A$ contained
    in $A^{(2)}$ satisfies the descending chain condition then \[\omega(A)=\omega(A/A^{(2)}).\]
\end{thm}

\begin{proof}
    Let $n=\omega(A/A^{(2)})$. Clearly it is enough to show that $\omega(A)\leq n$. To do this, fix an ideal $I$ of $A$
    contained in $A^{(2)}$ such that $\omega(A/I)=n$ and minimal with that property. Then choose elements $x_1,\dotsc,x_n\in A$
    such that the set $\{x_1+I,\dotsc,x_n+I\}$ generates $A/I$. Let $K$ denote the ideal of $A$ generated by
    the set $\{x_1,\dotsc,x_n\}$. Put $J = K \cap I$.
    
    We claim that $J=I$ and then $I\subseteq K$ together with $A=I+K$ imply $A = K$, which clearly leads to $\omega(A)\leq n$, as desired.
    To prove our claim assume, on the contrary, that $J \neq I$. Since $I/J \cong A/ K$ and $\omega(A) < \infty$, we obtain
    $\omega(I/J) < \infty$. Therefore, there exists an ideal $L$ in $I$ containing $J$ such that $\omega(I/L) = 1$.
    Because $A/J \cong I/J \times K/J$ and $J \subseteq L$, it follows that $L$ is an ideal of $A$.
    
    Since $I \subseteq A^{(2)}$, we get $I/L \subseteq (A/L)^{(2)}$. Moreover, as $A/L \cong I/L \times K/J$, we obtain
    $I/L \subseteq (I/L)^{(2)}$, which shows that $I/L$ is perfect.
    
    Because $\omega(I/L) = 1$, there exists $g \in I$ such that the ideal generated by $\overline{g}$ in $I/L$ is equal to $I/L$.
    Write $\overline{x}_1$ for the image of $x_1$ in $K/J$. If $\overline{y} \in I/L$ then $A/L \cong I/L \times K/J$ implies that
    $\overline{y}*(\overline{g}+\overline{x}_1)=\overline{y}*\overline{g}$. Thus the ideal $\overline{M}_1$
    generated by the set $\{\overline{y}*\overline{g}:\overline{y}\in I/L\}$ is contained in the ideal of $A/L$ generated by
    $\overline{g}+\overline{x}_1$. Furthermore, as $\overline{g}+ \overline{x}_1=\overline{g}\circ \overline{x}_1$ in $A/L$,
    it follows that \[(\overline{g}+\overline{x}_1)*\overline{y}=\overline{g}*(\overline{x}_1*\overline{y})
    +\overline{x}_1*\overline{y}+\overline{g}*\overline{y} = \overline{g}*\overline{y}.\]
    This shows that the ideal $\overline{M}_2$ generated by the set $\{\overline{g}*\overline{y}:\overline{y} \in I/L\}$
    is contained in the ideal generated by $\overline{g}+\overline{x}_1$. Hence $(I/L)/(\overline{M}_1+\overline{M}_2)$
    is a trivial skew left brace and cyclic as a group (with a generator the natural image of $\overline{g}$).
    Because $I/L$ is perfect, we get $I/L = \overline{M}_1+\overline{M}_2$. Therefore, $I/L$ is contained in the ideal $N$ of $A/L$
    generated by $\overline{g}+\overline{x}_1$. This in turn shows that $I = N + L $. Let $T = \{g+x_1,x_2,\dotsc,x_n\}$ and $M$ denotes
    the ideal generated by $T$. Then \[A = K+I = M +I = M + L,\] which yields $\omega(A/L) \leq \omega(A/I)$.
    Since $L$ is properly contained in $I$, we conclude that $\omega(A/L)=\omega(A/I)=n$, which is in contradiction with the minimality of $I$.
\end{proof}



\begin{thm}\label{thm:square-free}
    Let $A$ be an Artinian skew left brace of finite weight. Then \[\omega(A)=\omega((B/B^{(2)})_{\ab}),\] where $B=A/\Rad(A)$.
    In particular, if $(B/B^{(2)})_{\ab}$ is cyclic (for example if $A$ is of square-free order) then $\omega(A)=1$. 
\end{thm}

\begin{proof}
    Proposition~\ref{pro:weight} yields $\omega(A)=\omega(B)$. Furthermore, Theorem~\ref{thm:FP} implies
    $B\cong S_1\times\cdots\times S_n$, where $S_1,\dotsc,S_n$ are simple skew left braces. Therefore, Theorem~\ref{thm:Kutzko} gives
    $\omega(B)=\omega(B/B^{(2)})$. Note that \[B/B^{(2)}\cong S_{i_1}\times\cdots\times S_{i_k},\] where $\{i_1,\dotsc,i_k\}\subseteq\{1,\dotsc,n\}$
    and $S_{i_1},\dotsc,S_{i_k}$ are simple trivial skew left braces, i.e., these are simple groups. As $B/B^{(2)}$ is a trivial skew left brace, $\omega(B/B^{(2)})$ is the group-theoretic weight of $B/B^{(2)}$. Hence, by Kutzko's theorem~\cite{MR399272}, both weights are equal to the group-theoretic weight of the abelianization $(B/B^{(2)})_{\ab}$.
    
    Therefore, it remains to show that if $A$ is of square-free order then $(B/B^{(2)})_{\ab}$ is cyclic.
    But in this case \[(B/B^{(2)})_{\ab}\cong C_{p_1}\times\cdots\times C_{p_j}\cong C_{p_1\dotsm p_j}\]
    for some $j\leq k$ and some, pairwise distinct, prime numbers $p_1,\dots,p_j$.
\end{proof}

\begin{pro}\label{pro:omega(AxB)}
    Let $A$ be a perfect skew left brace such that $\omega(A)=1$ (this condition is satisfied if $A$ is Artinian of finite weight,
    in particular if $A$ is finite; see Theorem \ref{thm:Wiegold}). If $B$ is a trivial skew left brace of finite weight then
    $\omega(A\times B)=\omega(B)$.
\end{pro}

\begin{proof}
    Assume that $A=(a)$ and $B=(b_1,\dots,b_n)$, where $n=\omega(B)$. Fix $b\in B$. Since $B$ is a trivial skew left brace,
    for each $x\in A$ there exist $k\in\Z$ such that $(x,b^k)$ belongs to the ideal of $A\times B$ generated by $(a,b)\in A\times B$.
    Furthermore, \[(x,b^k)*(y,b^l)=(x*y,0)\] for $y\in A$ and $l\in\Z$. As $A$ is perfect, it follows
    that $A\times\{0\}$ is contained in the ideal of $A\times B$ generated by $(a,b)$. In particular, $A\times\{0\}$ is contained in the ideal $I$
    of $A\times B$ generated by $(a,b_1),\dotsc,(a,b_n)\in A\times B$. But then $(a,0)\in I$ leads to $(0,b_i)\in I$ for each
    $1\leq i\leq n$, which yields $A\times B=I$. Thus $\omega(A\times B)\leq n$. Since $(A\times B)/A\cong B$, we get the
    opposite inequality $n\leq\omega(A\times B)$. Hence the result follows.
\end{proof}

Recall that a skew left brace $B$ acts on a skew left brace $A$ if there is a group homomorphism $\theta\colon(B,\circ)\to\Aut(A,+,\circ)$.
In this case, one defines the semidirect product $A\rtimes B$ of skew left braces by declaring
\begin{align*}
    (a_1,b_1) + (a_2,b_2) & =(a_1+a_2,b_1+b_2),\\
    (a_1,b_1) \circ (a_2,b_2) & =(a_1\circ\theta(b_1)(a_2),b_1\circ b_2)
\end{align*}
for $a_1,a_2\in A$ and $b_1,b_2\in B$.
As an another application of Theorem~\ref{thm:Kutzko} we get the following strengthening
of Proposition~\ref{pro:omega(AxB)} for Artinian skew left braces.

\begin{cor}
    Let $A$ and $B$ be Artinian skew left braces. Assume that $B$ acts on $A$ and that $A$ is perfect. Then  $\omega(A\rtimes B)=\omega(B)$.
\end{cor}

\begin{proof}
    Let $a,c \in A$ and $b,d \in B$. Then
    \begin{align*}
        (a,b)*(c,d) & =-(a,b)+(a,b)\circ(c,d)-(c,d)\\
        & = (-a+a\circ\theta(b)(c)-c,b*d).
    \end{align*}
    Clearly, $A= A^{(2)} \subseteq (A \rtimes B)^{(2)}$ and thus $B^{(2)} \subseteq (A \rtimes B)^{(2)}$. This shows that
    $(A \rtimes B)^{(2)} = A \rtimes B^{(2)}$. By Theorem \ref{thm:Kutzko} we get
    \begin{align*}
        \omega(A \rtimes B) &= \omega((A \rtimes B)/(A\rtimes B)^{(2)})\\
        &= \omega((A\rtimes B)/(A\rtimes B^{(2)}))\\
        &= \omega(B/B^{(2)})\\
        &= \omega(B),
    \end{align*}
    and the result follows.
\end{proof}



Finally, Theorem~\ref{thm:Kutzko} leads also to a very short proof of Theorem~\ref{thm:Wiegold}.
Indeed, as $A$ is perfect, we obtain $A^{(2)} = A$. Since $A$ is Artinian, it follows by Theorem \ref{thm:Kutzko}
that $\omega(A)=\omega(A/A^{(2)})=\omega(0)=1$.

\section{Applications of the radical and weight of a skew brace}\label{sec:strbr}

In this section we provide some applications of the notion of weight of a skew left brace to study structure groups,
treated as skew left braces, of solutions of the YBE.

We shall start with properties concerning chain conditions. Recall that a skew left brace $A$ is said to be Noetherian if every ascending
chain of ideals of $A$ is eventually stationary. It is clear that a skew left brace is Noetherian if and only if all its ideals
have finite weight.

As was mentioned in Section \ref{sec:prel}, the structure group $G(X,r)$ of a finite non-degenerate solution $(X,r)$ of the YBE
is actually a skew left brace with the additive structure isomorphic to $(A(X,r),+)$. Moreover, the map
\[\lambda\colon (G(X,r),\circ)\to\Aut(A(X,r),+),\quad a\mapsto\lambda_a,\quad\lambda_a(b)=-a+a\circ b\]
is a homomorphism of groups. Since its image \[\mathcal{G}(X,r)=\langle\sigma_x:x\in X\rangle\]
(the so-called permutation group of $(X,r)$) is finite, $\Ker\lambda$ is of finite index in $G(X,r)$. Moreover, by
\cite[Theorem 2.7]{JKVA}, it follows that the group $A(X,r)$ is central-by-finite. Hence \[\Soc(G(X,r))=\Ker\lambda\cap\Cen(A(X,r))\]
is a finite index subgroup of $\Ker\lambda$. Therefore, $\Soc(G(X,r))$ is of finite index in $G(X,r)$. Since $G(X,r)$ is finitely generated,
it follows that $\Soc(G(X,r))$ is finitely generated as well. Thus $\Soc(G(X,r))$ is a finitely generated infinite abelian group.
This implies the following result.

\begin{pro}\label{pro:noeth}
    Let $(X,r)$ be a finite non-degenerate solution of the YBE.
    \begin{enumerate}
        \item The skew left brace $G(X,r)$ is not Artinian.
        \item The skew left brace $G(X,r)$ is Noetherian. 
    \end{enumerate}
\end{pro}

\begin{pro}\label{pro:w}
	Assume that $(X,r)$ is a finite non-degenerate solution of the YBE with $m$ orbits.
	Then \[\omega(G(X,r)/\Soc(G(X,r)))\leq \omega(G(X,r))\leq m.\] 
    Moreover, the solution $(X,r)$ is trivial if and only if $\omega(G(X,r))=|X|$.
    \begin{proof}
        The multiplicative group of $G(X,r)$ is generated by $X$. Hence it is generated, as an ideal, by a complete set of
        representatives of the $m$ orbits, which implies that $\omega(G(X,r))\leq m$. The lower bound for $\omega(G(X,r))$
        is obvious.
       
        Let $n=|X|$. If $(X,r)$ is trivial then $G(X,r)\cong \Z^n$ and thus we get $\omega(G(X,r))=\omega(\Z^n)=n$.
        Now assume that $\omega(G(X,r))=n$. Then $X$ is a minimal generating set of $G(X,r)$. We claim that $\lambda_x(y)=y$
        for all $x,y\in X$. Otherwise, if $\lambda_x(y)=z$ for some $y\ne z\in X$ then $X\setminus\{z\}$ is a generating set
        of size $<n$. The minimality of $X$ implies that $X$ embeds into $G(X,r)$ and hence
        \[r(x,y)=r_{G(X,r)}(x,y)=(\lambda_x(y),\lambda_x(y)'\circ x\circ y)=(y,y'\circ x\circ y).\]
        If there exist $u,v\in X$ such that $v'\circ u\circ v=w$ for some $u\ne w\in X$ then $X\setminus\{w\}$ is a generating
        set of size $<n$. Therefore, $r(x,y)=(y,x)$ for all $x,y\in X$.
    \end{proof}
\end{pro}

Note that if $(X,r)$ is a finite non-degenerate involutive solution of the YBE then $G(X,r)/\Soc(G(X,r))\cong\mathcal{G}(X,r)$.
Hence, by Proposition~\ref{pro:w}, we obtain $\omega(\mathcal{G}(X,r))\leq\omega(G(X,r))$. Furthermore, if $(X,r)$
is also indecomposable then $\mathcal{G}(X,r)$ and $G(X,r)$ both have weight equal to one.

\begin{exa}
    Let $(X,r)$ be a finite non-degenerate solution of the YBE. Then the map $\deg\colon (G(X,r),+)\to\Z$, given by $x\mapsto 1$ for $x\in X$,
    is a group homomorphism. Since $\deg(x*y)=0$ for all $x,y\in G(X,r)$, it follows that $G(X,r)$ is not perfect. 
 \end{exa}

For a skew left brace $A$, \[\Fix(A)=\{a\in A:\lambda_b(a)=a\text{ for all }b\in A\}\] is a left ideal of $A$ (see \cite{JKVAV}).
The following result can be understood as a brace-theoretic analog of Schur's theorem, which says that the derived subgroup
$[G,G]$ of a group $G$ is finite provided $\Cen(G)$ is of finite index in $G$; the converse is true if $G$ is finitely generated.

\begin{thm}\label{thm:Schur}
    Let $A$ be a skew left brace such that $A/\Ann(A)$ is finite. Then $A^{(2)}$ and $[A,A]_{+}$ are finite.
    The converse holds if the additive group $(A,+)$ is finitely generated.
\end{thm}

\begin{proof}
    Assume first that $A/\Ann(A)$ is finite. Let $X=(A,+)$ and $\mathcal{G}=\{\lambda_a:a \in A\}$. Note that
    \[|\mathcal{G}|\leq |A/\Soc(A)|\leq |A/\Ann(A)|<\infty.\]
    Next, consider the group $\Gamma = X \rtimes \mathcal{G}$, where the action of $\mathcal{G}$ on $X$ is given by applying the map
    $\lambda$. As $\Ann(A)\subseteq\Fix(A) \cap \Cen(X)$, it follows that $Z= \Ann(A) \times \{\id\}$ is a central subgroup of
    $\Gamma$. Because \[|\Gamma/Z|=|A/\Ann(A)|\cdot|\mathcal{G}|<\infty,\] we get $|\Gamma/\Cen(\Gamma)|<\infty$. Hence Schur's theorem
    implies that $[\Gamma,\Gamma]$ is a finite group. Moreover, if $a,b\in A$ then
    \[[(0,\lambda_a),(b,\id)]=(\lambda_a(b)-b,\id)=(a*b,\id)\in[\Gamma,\Gamma].\]
    This shows that $A^{(2)}$ is finite. Furthermore, as $[X,X]$ embeds in $[\Gamma,\Gamma]$, we conclude that $[X,X]$ is finite.
    Hence the first part follows.
    
    To prove the second part, assume that both $A^{(2)}$ and $[A,A]_+$ are finite. Moreover, assume that
    $x_1,\dotsc,x_n\in A$ generate the group $(A,+)$. If $a \in A$ and $b \in \Ann(A)$ then, by \eqref{*}, we get
    \[(a+b)*x_i = (b\circ a)*x_i = b*(a*x_i) + (a*x_i) + (b*x_i) = a*x_i\] and \[ x_i*(a+b) = x_i*a +a + x_i*b - a = x_i*a\]
    for each $1 \leq i \leq n$. Moreover, $\Ann(A) \subseteq\Cen(A,+)$ implies that $[a+b,x_i]_+ = [a,x_i]_+$ for each $1\leq i\leq n$.
    Hence, the map \[f\colon A/\Ann(A)\to(A^{(2)})^n\times(A^{(2)})^n \times([A,A]_+)^n,\] given by the formula
    \[f(\overline{a})=(a*x_1,\dotsc,a*x_n,x_1*a,\dotsc,x_n*a,[a,x_1]_+,\dotsc,[a,x_n]_+),\] is well-defined.
    We claim that $f$ is injective. To prove this let $a,b \in A$ be such that $f(\overline{a}) = f(\overline{b})$.
    Then $a*x_i = b*x_i$, $x_i*a=x_i*b$ and $[a,x_i]_+ = [b,x_i]_+$ for each $1\leq i \leq n$. Therefore,
    \begin{align*}
        x_i*(-a+b) &= x_i*(-a) -a+x_i*b + a \\
        & = x_i*(-a)-a+x_i*a + a\\
        & = x_i*(-a+a)\\
        & = 0,
    \end{align*}
    which yields $-a+b \in \Fix(A)$. Hence
    \begin{align*}
        a'\circ b & = a'\ + \lambda_{a'}(b)= a' + \lambda_{a'}(a -a + b)\\
        & = a' + \lambda_{a'}(a) -a+b = -a+b.
    \end{align*}
    Moreover, as $a*x_i = b*x_i$, we get $\lambda_a(x_i) = \lambda_b(x_i)$. Since this holds for all additive generators $x_1,\dotsc,x_n$
    of $A$, it follows that $\lambda_a = \lambda_b$. Therefore, $a' \circ b = - a +b \in \Ker\lambda$.
    Lastly, we claim that $-a + b \in\Cen(A,+)$. Indeed, as $[a,x_i]_+ = [b,x_i]_+$, it follows that
    \begin{align*}
        [-a+b,x_i]_+ & = -a+b+x_i-b+a-x_i\\
        & = -a+[b,x_i]_+ - [a,x_i]_+ +a = 0.
    \end{align*}
    In conclusion, we get $-a+b\in\Fix(A)\cap\Soc(A)=\Ann(A)$, which yields $\overline{a}=\overline{b}$.
    Because both $A^{(2)}$ and $[A,A]_+$ are finite, we conclude by injectivity of $f$ that $A/\Ann(A)$ is finite as well.
\end{proof}

\section{Torsion in structure groups}\label{sec:torsion}

Let $(X,r)$ be a finite non-degenerate solution of the YBE. In this section we study the torsion of the multiplicative group of $G(X,r)$.
It is known that if $(X,r)$ is involutive then $G(X,r)$ is torsion-free (see \cite{MR1637256}). However, as Example~\ref{exa:tf}
shows, the group $G(X,r)$ may be torsion-free also for non-involutive solutions.
In this section we determine precisely when the structure group is torsion-free.

\begin{exa}\label{exa:tf}
	Let $X=\{1,2,3,4\}$ and $r(x,y)=(\sigma(y),\tau(x))$, where $\sigma=(12)$ and $\tau=(34)$. Then $(X,r)$ is a non-involutive
    non-degenerate solution. Moreover, the skew left brace $G(X,r)$ is trivial and both of its groups (additive and multiplicative)
    are isomorphic to $\Z\times\Z$. 
\end{exa}

Recall that if $(X,r)$ is a solution of the YBE then  the derived solution $(X,r_\triangleright)$ of $(X,r)$ is the solution defined
as $r_\triangleright(x,y)=(y,y\triangleright x)$, where \[y\triangleright x=\sigma_y(\tau_{\sigma_x^{-1}(y)}(x)).\] The derived solution
$(X,r_\triangleright)$ is said to be indecomposable if the action of the group, generated by all permutations $x\mapsto y\triangleright x$,
on $X$ is transitive. Moreover, it is called a quandle provided $x\triangleright x=x$ for each $x\in X$.

\begin{thm}\label{thm:torsion}
	Let $(X,r)$ be a non-degenerate finite solution of the YBE. If the derived
	solution $(X,r_\triangleright)$ is indecomposable then $G(X,r)*G(X,r)$ is
	finite. Moreover, if $G(X,r)\ncong\Z$ then the multiplicative group of
	$G(X,r)$ has torsion. 
\end{thm}


\begin{proof}
	Let $A$ be the additive group of the skew left brace $G(X,r)$. As before,
	let $\deg\colon A\to\Z$ denote the homomorphism of groups given by
	$\deg(x)=1$ for all $x\in X$. Then, by \cite[Lemma 2.2]{MR3671570}, we have
	$\Ker(\deg)=[A,A]$. Moreover, recall that $A$ is central-by-finite (see
	\cite[Theorem 2.7]{JKVA}). Thus, by Schur's theorem, $[A,A]$ is finite.  As
	$G(X,r)*G(X,r)\subseteq\Ker(\deg)=[A,A]$, it follows that $G(X,r)*G(X,r)$
	is finite.
    
	To prove the second claim, suppose that $G(X,r)\ncong\Z$ and assume
	$G(X,r)*G(X,r)=0$. Then the skew left brace $G(X,r)$ is trivial.  Hence the
	multiplicative group of $G(X,r)$ is isomorphic to $A$. If $[A,A]=0$ then
	$A$ is abelian, and thus finitely generated free abelian group. Since, by
	assumption, the derived solution $(X,r_\triangleright)$ is indecomposable,
	it follows that $G(X,r)\cong A\cong\Z$, a contradiction.  In this case, the
	multiplicative group of $G(X,r)$ has a non-trivial finite subgroup
	isomorphic to $[A,A]$.
\end{proof}

\begin{rem}
	In the proof of Theorem~\ref{thm:torsion} we have used \cite[Lemma
	2.2]{MR3671570}, which holds for racks and not only for quandles.
\end{rem}

An example of a solution of the YBE satisfying all the requirements of
Theorem~\ref{thm:torsion} is the following.

\begin{exa}
    Let $X=C_3$ be the cyclic group of size three and \[r(x,y)=(xy^{-1}x^{-1},xy^2)=(y^2,xy^2).\] Then $(X,r)$ is a non-degenerate
    finite solution of the YBE with the multiplicative group of $G(X,r)$ isomorphic to \[G=\langle x,y,z\mid xy=yx=z^2,\,xz=zx=y^2\rangle,\]
    and the additive group of $G(X,r)$ isomorphic to \[A=\langle a,b,c\mid ab=bc=ca,\,cb=ba=ac\rangle.\]
    Routine calculations show that $G\cong\Z\times C_3$, so it has torsion. An element of order three in $G$ is $yz^{-1}$.
\end{exa}

\begin{thm}
    Let $(X,r)$ be a finite non-degenerate solution of the YBE. Then the group $G(X,r)$ is torsion-free if and only if
    the injectivization of $(X,r)$ is an involutive solution.
\end{thm}

\begin{proof}
    Clearly, we may assume that $(X,r)$ is an injective solution.
    
    Suppose $G=G(X,r)$ is torsion-free. Recall that $G(X,r)$ has a natural skew left brace structure with $(G,+)\cong A(X,r)$.
    Denote by $[G,G]_+$ the additive commutator subgroup of $(G,+)$. Since $A(X,r)$ is central-by-finite,
    by Schur's theorem $[G,G]_+$ is finite. As $[G,G]_+$ is a characteristic subgroup of $(G,+)$, it follows that
    $[G,G]_+$ is a subgroup of $(G,\circ)$. Because $G$ is torsion-free, this subgroup has to be trivial. Hence,
    $(G,+) \cong A(X,r)$ is abelian. Therefore, by \cite[Theorem 7.15]{MR3974961}, $(G,+)$ is free abelian.
    Thus, because of the injectivity of $(X,r)$, using similar arguments as in the last paragraph of the proof of
    \cite[Theorem 2.7]{JKVA}), we conclude that the solution $(X,r)$ is involutive.
    The reverse implication is well-known \cite{MR1637256}, hence the result is proved.
\end{proof}







We do not know how to compute the weight of the structure skew brace of a
finite solution of the YBE. Moreover, is there a connection between the weight
of a left brace $A$ and the weight of the left brace $G(A,r_A)$, where
$(A,r_A)$ is the solution associated to $A$?

\subsection*{Acknowledgments}
The first author is supported in part by Onderzoeksraad of Vrije Universiteit
Brussel and Fonds voor Wetenschappelijk Onderzoek (Flanders), grant G016117.
The second author is supported by Fonds voor Wetenschappelijk Onderzoek
(Flanders), grant G016117. The third author is supported by Fonds voor
Wetenschappelijk Onderzoek (Flanders), via an FWO Aspirant-mandate.  The fourth
author is supported in part by PICT 2016-2481 and UBACyT 20020170100256BA.
Vendramin acknowledges the support of NYU-ECNU Institute of Mathematical
Sciences at NYU Shanghai. Thanks go to Ferran Ced\'o for helpful comments.

\bibliographystyle{abbrv}
\bibliography{refs}

\end{document}